\documentclass{elsarticle}

\usepackage[latin1]{inputenc}
\usepackage[all,cmtip]{xy}
\usepackage{array}
\usepackage{amsmath, amssymb}
\usepackage{graphics}
\usepackage{rotating}
\usepackage{caption}
\usepackage{subcaption}
\usepackage{amsthm}
\usepackage{algorithm}
\usepackage{algorithmic}
\usepackage{booktabs}
\usepackage{color}
\usepackage{cancel}
\usepackage{caption}
\usepackage{subcaption}
\usepackage{pdflscape}
\usepackage{epstopdf}
\usepackage{subcaption}
\usepackage{multirow}
\usepackage{graphicx}
\usepackage[hidelinks]{hyperref}
\usepackage{multicol}
\usepackage{float}
\usepackage[top=4.4cm,left=4.7cm, right=4.75cm, bottom=4.2cm]{geometry}

\def\deg{{\rm deg}}

\newtheorem{theorem}{{\bf Theorem}}
\newtheorem{remark}{{\bf Remark}}

\newtheorem{proposition}[theorem]{{\bf Proposition}}
\newtheorem{lemma}[theorem]{{\bf Lemma}}
\newtheorem{example}{{\bf Example}}

\def\bfb{{\boldsymbol{b}}}

\def\bfp{{\boldsymbol{p}}}
\def\bfq{{\boldsymbol{q}}}

\def\bfx{{\boldsymbol{x}}}

\def\bfQ{{\boldsymbol{Q}}}

\def\RR{\mathbb{R}}

\begin{document}

\begin{frontmatter}
\title{Computing the symmetries of a ruled rational surface.}

%%%%%%%%%%%%%%%%%%%%%%%%%%%%%%%%%%%%%%%%%%%%%%%%%%%%%%%%%%%%%%%%%%%%%%%%%%%%%%%%%%%%%%%%%%%%%%%%%
%% NOTA PARA MI: CUANDO VAYAMOS A ESCRIBIR PAPER, QUITAR COMENTARIOS DESDE AQUI...................

\author[a]{Juan Gerardo Alc\'azar\fnref{proy,proy2}}
\ead{juange.alcazar@uah.es}
\author[a]{Emily Quintero\fnref{proy3}}
\ead{emily.quintero@edu.uah.es}

\address[a]{Departamento de F\'{\i}sica y Matem\'aticas, Universidad de Alcal\'a,
E-28871 Madrid, Spain}

%\address[b]{SINTEF IKT, PO Box 124 Blindern, 0314 Oslo, Norway, and\\Department of Mathematics, University of Oslo, PO Box 1053, Blindern, 0316 Oslo, %Norway\!\!\!\!\!\!\!\!

\fntext[proy]{Supported by the Spanish ``Ministerio de
 Econom\'ia y Competitividad" under the Project MTM2014-54141-P.}

\fntext[proy2]{Member of the Research Group {\sc asynacs} (Ref. {\sc ccee2011/r34}) }

\fntext[proy3]{Supported by a grant from the Carolina Foundation.}

%%%% HASTA AQUI...............................................................
%%%%%%%%%%%%%%%%%%%%%%%%%%%%%%%%%%%%%%%%%%%%%%%%%%%%%%%%%%%%%%%%%%%%%%%%%%%%%%

\begin{abstract}
We present a method for computing all the symmetries of a rational ruled surface defined by a rational parametrization which works directly in parametric rational form, i.e. without computing or making use of the implicit equation of the surface. The method proceeds by translating the problem into the parameter space, and relies on polynomial system solving. If we want all the symmetries of the surface, including rotational symmetries, we need to deal with polynomial systems in four variables; if we are only interested in involutions (e.g. central symmetries, axial symmetries, reflections in a plane), we can come down to bivariate polynomial systems. An application to compute symmetries of an implicit algebraic surface under certain conditions is also provided. 
\end{abstract}

%\maketitle
 \end{frontmatter}

\section{Introduction}\label{section-introduction}

Symmetry is a property commonly found in nature and in manufactured items, usually associated with the ideas of beauty and proportion. In Geometry, and in particular when studying algebraic surfaces, symmetry arises frequently, associated with the same idea of beauty. 

Symmetries of 3-space, also called isometries, are well classified \cite{Coxeter}, and comprise translations, central symmetries, reflections in a plane, rotational symmetries (with axial symmetries as a special case), and their compositions. A symmetry of a surface is a symmetry of 3-space that leaves the surface invariant. In particular, symmetries of 3-space are orthogonal transformations. 

Knowing the symmetries of a surface is useful in order to understand the geometry of the surface and visualize the surface correctly. It is also useful in applications, for instance image storage, object detection and recognition, or medial axis computations. In fact, in the literature of applied fields like Computer Aided Geometric Design, Pattern Recognition or Computer Vision one can find many methods to detect symmetries (see for instance the Introduction to \cite{AH15}), although these methods are usually applied to objects where no specific structure is assumed, and are more orientented towards finding ``approximate" symmetries, in some sense. 

In contrast, in this paper we address a type of surfaces with a strong structure, namely rational ruled algebraic surfaces, and we make use of the structure of the surface in order to compute its symmetries. Additionally, we assume that the ruled surface to be analyzed is given by a rational parametrization in standard form. 

A naive approach to compute the symmetries of such an object is to compute first the implicit equation of the surface, which can be efficiently done by using $\mu$-bases \cite{Chen1, Chen2}, to pick a generic orthogonal transformation, and finally to impose that the orthogonal transformation leaves the surface invariant. This approach leads to a polynomial system in the degree $N$ of the implicit surface, with 12 variables, the parameters of the orthogonal transformation. The resulting polynomial system, however, is typically very big, so this approach is often impractical. Nevertheless, to our knowledge this method was so far the only alternative to solve the problem considered here. 

In this paper we use a different approach, inspired in previous works on symmetries of rational curves \cite{AHM15, HJ18} and polynomially parametrized surfaces \cite{AH15}. The idea is to reduce the problem to computations in the parameter space. In order to do this, we observe that whenever the parametrization defining the surface is \emph{proper}, i.e. birational, any symmetry of the surface comes from a birational transformation of the parameter space (the plane), i.e. a Cremona transformation. Taking advantage of the structure of the surface, in our case of the fact that the surface is ruled, we prove that the corresponding birational transformation of the parameter space has a predictable form, which makes possible to compute it. From here, the symmetry itself can be found.

From a computational point of view, our method relies, as in the naive approach mentioned before, on polynomial system solving. However, on one hand our method works directly with the rational parametrization of the surface, so we do not need to compute or make use of the implicit equation of the surface. On the other hand, the polynomial system we end up with has either 4 variables, in case that we are interested in all the symmetries of the surface, including rotational symmetries, or 2 variables, if we are only interested in involutions (e.g. central symmetries, axial symmetries, reflections in a plane). 

Additionally, we show that our results can be applied to implicit algebraic surfaces under certain conditions. In order to do this, we take advantage of two facts: first, the form of highest degree of the implicit equation of the surface always defines a conic surface, which in particular is a ruled surface. Second, knowing the symmetries of the surface defined by the form of highest degree provides clues on some of the symmetries of the entire surface. 

The structure of the paper is the following. We start with a preliminary section, Section \ref{gen-surf}, where we fix the hypotheses required on the input, together with some general notions to be used later on the paper. The main ideas behind the method, and the computational method itself, are presented in Section \ref{sec-symmetries}. Section \ref{sec-exp} reports on experimentation results obtained with the help of the computer algebra system Maple 18. Section \ref{app} presents an application of our ideas to implicit algebraic surfaces, under certain conditions. We close with a Conclusion section in Section \ref{sec-conclusion}.

\section{Preliminaries.} \label{gen-surf}

Let $S$ be a real ruled surface, defined by means of a real, rational parametrization in so-called ``standard form"
\begin{equation}\label{surf}
\pmb{x}(t,s)=\pmb{p}(t)+s\cdotp\pmb{q}(t).
\end{equation}
At each point $P=\pmb{x}(t,s)\in S$ the vector $\pmb{q}(t)$ defines the direction of the \emph{ruling} through $P$, i.e. the line $L_P$ through $P$ contained in $S$. In this paper we will assume that $S$ is not cylindrical, so that $\pmb{q}(t)$ is not a multiple of a constant vector, and that $S$ is not a surface of revolution. One can detect cylindrical surfaces and surfaces of revolution using the results in \cite{AG17}. Furthermore, if $S$ is either cylindrical or a surface of revolution, the symmetries of $S$ can be computed by using ideas in \cite{AH15}: see \cite[\S 2.2.4]{AH15} for surfaces of revolution, and \cite[\S 4]{AH15} for cylindrical surfaces. In particular, if $S$ is neither a surface of revolution nor a cylindrical surface, then the number of symmetries of $S$ is finite and translations are not involved (see \cite{AH15}). 

Furthermore, we will also suppose that $S$ is not doubly-ruled, i.e. that there are not two different families of rulings contained in $S$. It is well-known that the doubly-ruled surfaces are the plane, the hyperbolic paraboloid, and the single-sheeted hyperboloid. For paraboloids and hyperboloids one can find the symmetries of the surface by first computing the implicit equation, which is easy to do in the case of quadrics, and then performing easy calculations on the matrix defining the quadric (essentially, computing the eigenspaces of the matrix).

In the rest of the paper we will assume  that $\bfx(t,s)$ is \emph{proper}, i.e. that the parametrization in Eq. \eqref{surf} is injective except at most a 1-dimensional subset of $S$; in particular, this implies that $\bfx^{-1}$ exists and is rational. Additionally, we need two more assumptions on Eq. \eqref{surf}. First, and this will be important in order to develop our results, we will suppose that $\bfq(t)$ is polynomially parametrized, so that the components of $\bfq(t)$ have no denominators. We will also suppose that the (polynomial) components of $\bfq(t)$ are relatively prime, i.e. writing $\bfq(t)=(q_1(t),q_2(t),q_3(t))$, we will assume that $\gcd(q_1(t),q_2(t),q_3(t))=1$. 

Let us see that we can always achieve the two last requirements ($\bfq(t)$ polynomial, with components relatively prime), so that the above assumptions can be considered as completely general. Indeed, if some of them do not hold, then we can replace 
\begin{equation} \label{replace}
\bfq(t):=\mu(t) \bfq(t),\mbox{ }\mu(t)=\frac{\mu_1(t)}{\mu_2(t)},
\end{equation}
where $\mu_1(t)$ is the least common multiple of the denominators of the components of $\bfq(t)$, and $\mu_2(t)$ is the greatest common divisor of the numerators of the components of $\bfq(t)$. Notice that since $\mu(t) \bfq(t)$ is parallel to $\bfq(t)$ for all $t$, the new parametrization $\widehat{\bfx}(t,s)=\bfp(t)+s \mu(t)\bfq(t)$ also defines the surface $S$, because the rulings of the surfaces defined by $\bfx(t,s)$ and $\widehat{\bfx}(t,s)$ coincide. Furthermore, we can do this without losing properness, as shown by the next lemma. 

\begin{lemma}\label{prop-new}
Let $S$ be a ruled surface, and let $\bfx(t,s)$ be a proper parametrization of $S$ defined by Eq. \eqref{surf}. Let $\widehat{\bfx}(t,s)=\bfp(t)+s \mu(t)\bfq(t)$, where $\mu(t)$ is defined as in Eq. \eqref{replace}. Then $\widehat{\pmb{x}}(t,s)$ is also proper.
\end{lemma}

\begin{proof} Suppose that $\widehat{\pmb{x}}(t,s)$ is not proper. Then a generic point $P$ of $S$ is generated via $\widehat{\pmb{x}}(t,s)$ by two different pairs $(t_1,s_1)\neq(t_2,s_2)$; furthermore, since $P$ is generic we can assume that $\mu(t_1)\cdot \mu(t_2)\neq 0$. In this situation, $P$ is generated via $\bfx(t,s)$ by the pairs $(t_1,\tilde{s}_1)$, $(t_2,\tilde{s}_2)$ where $\tilde{s}_i=\mu(t_i)\cdot s_i$, $i=1,2$. Since $\bfx(t,s)$ is proper by hypothesis, we deduce that $t_1=t_2$, $\tilde{s}_1=\tilde{s}_2$, in which case $\mu(t_1)=\mu(t_2)$ too. And since $\tilde{s}_1=\tilde{s}_2$ and $\mu(t_i)\neq 0$, we conclude that $s_1=s_2$, contradicting that the pairs $(t_1,s_1)$ and $(t_2,s_2)$ were different.
\end{proof}

An \emph{isometry} of $\RR^3$ is a map $f:\RR^3\longrightarrow \RR^3$ preserving Euclidean distances. Any isometry $f$ of $\RR^3$ has the form \begin{equation}\label{eq:isometry}
f({\bf x}) = \bfQ {\bf x} + \bfb, \qquad {\bf x}\in \RR^3,
\end{equation}
with $\bfb \in \RR^3$ and $\bfQ\in \RR^{3\times 3}$ an orthogonal matrix, i.e. $\bfQ^T \bfQ=I$, where $I$ denotes the $3\times3$ identity matrix. In particular, $\det(\bfQ) = \pm 1$. We say that $f$ is a \emph{symmetry} of a surface $S$ if $f$ is an isometry such that $f(S)=S$. The space of isometries of $ \RR^3$ forms a group under composition that is generated by \emph{reflections}, i.e., symmetries with respect to a plane. The nontrivial isometries of $\RR^3$ are reflections in a plane, rotations about an axis, translations, and their compositions. Composing three reflections in mutually perpendicular planes through an affine point $P$ yields a \emph{central symmetry} with center $P$, i.e., a symmetry with respect to the point $P$, called the {\it symmetry center}. 

Additionally, we say that $S$ has \emph{rotational symmetry} if there exist a line $\ell$ and a real $\theta\in (0,2\pi)$ such that $S$ is invariant under the rotation about $\ell$, by the angle $\theta$. In that case we say that $\ell$ is an \emph{axis of rotation} of $S$. The special case of rotation by the angle $\pi$ is called a \emph{half-turn} or an \emph{axial symmetry}. The axis of rotation in this case is called the {\it symmetry axis}. 

Finally, we say that $f:\RR^n\longrightarrow \RR^n$ is an \emph{involution}, if $f\circ f=\mbox{id}_{{\Bbb R}^n}$. For $n=3$, reflections in a plane, axial symmetries and central symmetries are involutions. 

%Additionally \cite{AH15}, if $S$ is an algebraic surface of degree $d$, invariant under a rotation about an axis $\ell$, by a non-trivial angle $\theta$ (i.e. $\theta\neq 2k\pi$, %with $k\in {\Bbb Z}$). If $\ell$ is not an axis of revolution of $S$, then $\theta=\frac{2\pi}{m}$, where $0<m\leq 2d$, $m\in {\Bbb N}$. As a consequence of these results, %rotations of surfaces which are not surfaces of revolution satisfy that $f^n=\mbox{id}_{{\Bbb R}^3}$ for some $n\in\NN$.

\section{Symmetries of ruled surfaces.}\label{sec-symmetries}

Let $S$ be a rational ruled surface parametrized as in Eq. \eqref{surf}, in the conditions of Section \ref{gen-surf}, i.e. $\bfx(t,s)$ is proper, $\bfq(t)$ is polynomial and has relatively prime components, and $S$ is neither cylindrical, nor doubly-ruled, nor a surface of revolution.  

\begin{theorem} \label{th-fundam}
 Let $S$ be a rational real ruled surface properly parametrized as in Eq. \eqref{surf}. An isometry $f:{\Bbb R}^3\to {\Bbb R}^3$, $f({\bf x})=\bfQ {\bf x}+\bfb$ is a symmetry of $S$ if and only if there exists a unique, birational transformation $\varphi:{\Bbb R}^2\to {\Bbb R}^2$, such that the diagram 
	\begin{equation}\label{eq:fundamentaldiagram}
	\xymatrix{
		S \ar[r]^{f} & S \\
		\RR^2 \ar@{-->}[u]^{\pmb{x}} \ar@{-->}[r]_{\varphi} & \RR^2 \ar@{-->}[u]_{\pmb{x}}
		}
	\end{equation}
	is commutative. In particular, for a generic point $(t,s)\in \RR^2$ we have 
	\begin{equation}\label{fundam-eq}  f\circ \bfx=\bfx\circ \varphi.
	\end{equation}
\end{theorem}
																																									
\begin{proof}
``$\Rightarrow$" Since $\bfx$ is proper by hypothesis, $\bfx^{-1}$ exists and is rational. Therefore, $\varphi=\bfx^{-1}\circ f \circ \bfx$ is birational, because it is the composition of birational transformations. ``$\Leftarrow$" Since $f\circ \pmb{x}=\pmb{x}\circ \varphi$, whenever $\bfx(t,s)$ and $(\bfx\circ \varphi)(t,s)$ are well-defined $(f\circ \bfx)(t,s)\in S$, so $f(S)\subset S$. Since $f$ is an isometry, $f(S)$ defines a rational surface, i.e. $f(S)$ does not degenerate into a curve. Additionally both $f(S),S$ are rational, and therefore irreducible; since $f(S)\subset S$ and $f(S),S$ are irreducible, $f(S)=S$, i.e. $f$ leaves $S$ invariant.
\end{proof}

\begin{remark}\label{rem-basepoints} It can happen that only one of the sides of Eq. \eqref{fundam-eq} is defined. For instance, let $S$ be the conical surface parametrized by $$\bfx(t,s)=s\cdotp(-t^4-6t^2+3,8t^3,(t^2+1)^2).$$ One can check that this surface is invariant by a rotation of $\frac{2\pi}{3}$ degrees about the $z$-axis, given by $f(x,y,z)=\bfQ\cdotp (x,y,z),$ where $$Q=\left( \begin{array}{rcl} -\frac{1}{2} & -\frac{\sqrt{3}}{2} & 0 \\ \frac{\sqrt{3}}{2} & -\frac{1}{2} & 0 \\ 0 & 0 & 1 \end{array}\right).$$ Furthermore, one can also check that in this case $f\circ \bfx=\bfx\circ \varphi$, with
\begin{equation}\label{this-var}
\varphi(t,s)= \left(\frac{-\sqrt{3}t - 3}{3t - \sqrt{3}} , \frac{9}{16}s\left(t - \frac{\sqrt{3}}{3}\right)^4 \right),
\end{equation}
holds for a generic point $(t,s)$ of the parameter space. Now take $(t_0,s_0)=\left(\frac{\sqrt{3}}{3},1\right)$, so $\bfx(t_0,s_0)=\left(\frac{8}{9},\frac{8\sqrt{3}}{9},\frac{16}{9}\right)$ and \[(f\circ \bfx)(t_0,s_0)=f\left(\frac{8}{9},\frac{8\sqrt{3}}{9},\frac{16}{9}\right)=\left(-\frac{16}{9},0,\frac{16}{9}\right).\]In particular, the left hand-side of Eq. \eqref{fundam-eq} is well defined. However, for $(t_0,s_0)=(\frac{\sqrt{3}}{3},1)$ the denominator of the first component of $\varphi(t_0,s_0)$ is zero, and $({\bfx}\circ \varphi)(t_0,s_0)$ is not defined. One can check that in this case, the symmetry $f$ maps the point $(\frac{8}{9},\frac{8\sqrt{3}}{9},\frac{16}{9})$, generated by $(t_0,s_0)$, to the point $(-\frac{16}{9},0,\frac{16}{9})$, which is a point of $S$ missed by the parametrization $\bfx$, i.e. not generated by $\bfx$ for any pair $(t,s)$.

If we consider the situation in a projective setup, we observe that Eq. \eqref{fundam-eq} fails at the projective point corresponding to $(t_0,s_0)=(\frac{\sqrt{3}}{3},1)$. Indeed, let us represent by $[t:s:\omega]$ the elements of the parameter space, which is now ${\Bbb P}^2({\Bbb R})$ (the last coordinate corresponds to the homogenization variable). Then the point $[\frac{\sqrt{3}}{3}:1:1]$ of the parameter space, corresponding to the affine point $(t_0,s_0)=(\frac{\sqrt{3}}{3},1)$, maps to $[576:0:0]$, which is a base point of 
\[
\widehat{\bfx}(t,s,\omega)=[s(-t^4-6t^2\omega^2+3\omega^4):8st^3\omega:s(t^2+\omega^2)^2:\omega^5],
\]
the parametrization of the projective closure $\widehat{S}$ of $S$. Thus, the left hand-side of Eq. \eqref{fundam-eq} is $[-\frac{16}{9}:0:\frac{16}{9}:1]$, while the right hand-side of Eq. \eqref{fundam-eq} is $[0:0:0:0]$.  
\end{remark}

Our strategy in order to find the symmetries of $S$ is to first find the transformations $\varphi$ in the bottom part of the diagram in Eq. \eqref{eq:fundamentaldiagram}, and then, from here, compute the symmetries themselves. Moreover, from Eq. \eqref{fundam-eq} one can easily see that each symmetry of $f$ is associated with a different $\varphi$. From Theorem \ref{th-fundam} we observe that the $\varphi$ are birational transformations of the plane. Such transformations are called \emph{Cremona transformations}. However, unlike the birational transformations of the line, which are the well-known 
\emph{M\"obius transformations}, i.e. the transformations of the type
\begin{equation}\label{eq:Moebius}
\psi: {\Bbb R} \dashrightarrow {\Bbb R}, \qquad \psi(t) = \frac{\alpha t + \beta}{\gamma t + \delta}, \qquad \alpha \delta - \beta \gamma \neq 0,
\end{equation}
Cremona transformations do not have a generic closed form. Because of this, in order to find out how $\varphi$ looks like, we need to make use of the properties of the surface we are investigating, in this case of the fact that $S$ is ruled. 

The following result provides a first clue in this direction.

\begin{proposition}\label{fund}
Let $S$ be a rational ruled surface properly parametrized as in Eq. \eqref{surf}, which is not doubly ruled. Let $f({\bf x})=\bfQ {\bf x}+\bfb$ be a symmetry of $S$, and let $\varphi:\RR^2\to \RR^2$ be the birational transformation making the diagram in Eq. \eqref{eq:fundamentaldiagram} commutative. Then
	\begin{equation}\label{phi-funct}
	\varphi(t,s)=(\psi(t),a(t)\cdotp s+ c(t)),
	\end{equation}
	where $\psi(t)$ is a M\"obius transformation and $a(t),c(t)$ are rational functions.
\end{proposition}
 
\begin{proof} Since $f$ is a symmetry of $S$, $f$ maps rulings of $S$ to rulings of $S$. Let $\varphi(t,s)=(\varphi_1(t,s),\varphi_2(t,s))$. A generic ruling of $S$ is defined by 
$\pmb{x}(t_a,s)$, where $t_a$ is constant. Since $S$ is not doubly ruled, $\pmb{x}(t_a,s)$ is mapped to $\pmb{x}(t_b,s)$, where $t_b$ is also a constant. Using Eq. \eqref{fundam-eq}, we get 
$$f(\bfx(t_a,s))=\bfx(\varphi(t_a,s))=\bfx(\varphi_1(t_a,s),\varphi_2(t_a,s))=\bfx(t_b,s),$$so $\varphi_1(t_a,s)=t_b$, i.e. $\varphi_1(t_a,s)$ does not depend on $s$. Since this happens for a generic $t_a$, we deduce that $\varphi_1(t,s)=\varphi_1(t)$. Since $\varphi$ is birational, $\varphi_1$ is birational as well; in particular, we deduce that $\varphi_1$ is a birational transformation of the line, so $\varphi_1$ must be a M\"obius transformation, that we represent by $\psi(t)$. The rest of the theorem follows from Eq. \eqref{fundam-eq}, taking into account that $f({\bf x})=\bfQ {\bf x}+\bfb$. 
\end{proof}

Let us now investigate the structure of the function $a(t)$ in Eq. \eqref{phi-funct}. In order to do this, recall that $\pmb{x}(t,s)=\bfp(t)+s\cdotp\pmb{q}(t)$, where $\bfq(t)=(q_1(t),q_2(t),q_3(t))$, each $q_i(t)$ is polynomial and $\gcd(q_1,q_2,q_3)=1$. Also, let 

\begin{equation}\label{eq-n}
n=\mbox{max}\{\deg(q_1(t)),\deg(q_2(t)),\deg(q_3(t))\},
\end{equation}

\noindent and let us write \[a(t)=\dfrac{A(t)}{B(t)},\mbox{ }\psi(t) = \dfrac{\alpha t + \beta}{\gamma t + \delta},\]where $A,B\in\RR[t]$, $\gcd(A,B)=1$, and $\alpha \delta - \beta \gamma \neq 0$.  Now combining Eq. \eqref{phi-funct} and Eq. \eqref{fundam-eq} with $f({\bf x})=\bfQ {\bf x}+\bfb$, we get
\begin{equation}\label{at}
		\bfQ\cdotp\bfq(t)=a(t)\cdotp\bfq(\psi(t)).
	\end{equation}
Since $\bfq(t)$ is polynomial the left hand-side of Eq. \eqref{at} is polynomial, so the right hand-side of Eq. \eqref{at} must be polynomial as well. This yields the following results; here, we denote the entries of the matrix $\bfQ$ by $\bfQ_{ij}$. 

\begin{lemma}\label{At}
	$(\gamma t + \delta)^n$ divides $A(t)$.
\end{lemma}

\begin{proof}
	From Eq. \eqref{at}, for $i=1,2,3$ we get 
	\begin{equation}\label{est}
	\bfQ_{i1}\cdotp q_1(t)+\bfQ_{i2}\cdotp q_2(t)+\bfQ_{i3}\cdotp q_3(t)=a(t)\cdotp q_i(\psi(t)),
	\end{equation} where $q_i(t)=a_{\ell_i} t^{\ell_i}+a_{\ell_i-1}t^{\ell_i-1}+\cdots+a_0,$ with $\ell_i\leq n$ for $i\in \{1,2,3\}$. Furthermore, $\ell_i=n$ for at least one $i\in \{1,2,3\}$. Additionally, $$q_i(\psi(t))=\dfrac{a_{\ell_i}(\alpha t+\beta)^{\ell_i}+a_{\ell_i-1}(\alpha t+\beta)^{\ell_i-1}(\gamma t+\delta)+\cdots+a_0(\gamma t+\delta)^{\ell_i}}{(\gamma t+\delta)^{\ell_i}}.$$ Since $\gamma t +\delta$ does not divide $\alpha t+\beta$, the numerator and denominator of $q_i(\psi(t))$ are relatively prime. Since the left hand-side of Eq. \eqref{est} is polynomial, $a(t)\cdotp q_i(\psi(t))$ must be polynomial as well, so $(\gamma t+\beta)^{\ell_i}$ divides $A(t)$. Since $\ell_i=n$ for some $i\in \{1,2,3\}$, the statement follows. 
\end{proof}

\begin{lemma}\label{Bt}
	$B(t)$ is a constant.
\end{lemma}

\begin{proof}
Let $N_i(t)$ be the numerator of $q_i(\psi(t))$, and recall that $\gcd(q_1,q_2,q_3)=1$. Since the left hand-side of Eq. \eqref{est} is polynomial, $B(t)|N_i(t)$ for $i=1,2,3$. Thus, $B(t)|G(t)$, where $G=\gcd(N_1,N_2,N_3)$. Now suppose that $G(t)$ is not constant. Then $N_1,N_2,N_3$ have a common root $t_0$. Moreover, since the numerators and denominators of the $q_i(\psi(t))$ are relative prime, $\gamma t_0+\delta\neq 0$. Therefore, $\psi(t_0)$ is well defined and $\psi(t_0)$ is a common root of the $q_i(t)$, because $q_i(\psi(t_0))=\frac{N_i(t_0)}{(\gamma t_0+\delta)^{\ell_i}}$. But this contradicts that $\gcd(q_1,q_2,q_3)=1$. Thus, $G(t)$ is constant and since $B(t)|G(t)$, $B(t)$ must be a constant.
\end{proof}

Finally, we get the following proposition on the form of the function $a(t)$.

\begin{proposition} \label{prop-a(t)}
The function $a(t)$ satisfies that $a(t)=k\cdotp(\gamma t+\delta)^n$, where $k$ is a constant, and $n=\mbox{max}\{\deg(q_1(t)),\deg(q_2(t)),\deg(q_3(t))\}$.
\end{proposition}

\begin{proof}From the two previous lemmas we have $a(t)=k(t)\cdotp(\gamma t+\delta)^n$ for some polynomial $k(t)$. Additionally, from Eq. \eqref{at}
\begin{equation}\label{eq}
	\bfQ\cdotp\bfq(t)=k(t)\cdotp(\gamma t+\delta)^n\cdotp\bfq(\psi(t)).
	\end{equation}
Since $\bfQ$ is an orthogonal transformation, taking norms we get 
\begin{equation}\label{with-norms}
\Vert \bfq(t)\Vert^2=k^2(t)\cdot (\gamma t+\delta)^{2n}\cdot \Vert \bfq(\psi(t))\Vert^2.
\end{equation}
Since $n=\mbox{max}\{\deg(q_1(t)),\deg(q_2(t)),\deg(q_3(t))\}$ the degree in $t$ of the left hand-side of Eq. \eqref{with-norms} is $2n$. Since the denominator of $\Vert\bfq(\psi(t))\Vert^2$ is $(\gamma t+\delta)^{2n}$ and the degree in $t$ of the numerator of $\Vert\bfq(\psi(t))\Vert^2$ is $2n$ as well, we deduce that the degree in $t$ of $k^2(t)$ must be zero. Therefore, $k(t)$ must be a constant. 
	\end{proof}
	
We summarize the previous results in the following theorem. 

\begin{theorem}\label{fund2}
Let $S$ be a rational ruled surface properly parametrized as in Eq. \eqref{surf}, which is not doubly ruled, let $\bfq(t)=(q_1(t),q_2(t),q_3(t))$, with $q_i(t)\in {\Bbb R}[t]$ for $i=1,2,3$, and let also $n=\mbox{max}\{\deg(q_1(t)),\deg(q_2(t)),\deg(q_3(t))\}$. Let $f({\bf x})=\bfQ {\bf x}+\bfb$ be a symmetry of $S$, and let $\varphi:\RR^2\to \RR^2$ be the birational transformation making the diagram in Eq. \eqref{eq:fundamentaldiagram} commutative. Then
	\begin{equation}\label{phi-funct-2}
	\varphi(t,s)=(\psi(t),k\cdotp(\gamma t+\delta)^n\cdotp s+ c(t)),
	\end{equation}
	where $\psi(t)$ is a M\"obius transformation, $k$ is a constant, and $c(t)$ is a rational function.
\end{theorem}

Hence, we have proven that in order to find the functions $\psi(t),a(t)$ in Eq. \eqref{phi-funct} it suffices to determine five parameters, namely the parameters of the M\"obius transformation $\psi(t)$, and the constant $k$. However, we can always come down to 4 parameters, either by fixing one of the parameters of $\psi(t)$ as 1, or by introducing $k$ into the bracket $(\gamma t +\delta)^n$. In our experimentation we have chosen the first option, since in general we have observed that it provides better timings. 

However, we can do better in case that we are only interested in the involutions of $S$. Recall that $f$ is an involution iff $f\circ f=\mbox{id}_{{\Bbb R}^3}$. From Eq. \eqref{fundam-eq}, one can see that $f\circ f=\mbox{id}_{{\Bbb R}^3}$ iff the corresponding $\varphi$ satisfies $\varphi\circ \varphi=\mbox{id}_{{\Bbb R}^2}$. Since
\[
\varphi(t,s)=(\varphi_1(t,s),\varphi_2(t,s))=(\psi(t),s\cdot k(\gamma t+\delta)^n+c(t)),
\]
imposing $(\varphi\circ \varphi)(t,s)=(t,s)$ one gets two possibilities:

\begin{itemize}
\item [(i)] $(\varphi_1\circ \varphi_1)(t,s)=t$, i.e. $(\psi\circ \psi)(t)=t$. In turn, this implies that 
\[
\alpha^2-\delta^2=0,\mbox{ }\beta(\alpha+\delta)=0,\mbox{ }\gamma(\alpha+\delta)=0.
\]
Therefore, either $\alpha=-\delta$, or $\alpha+\delta\neq 0$ and $\alpha=\delta$, $\beta=\gamma=0$. 
\item [(ii)] $\varphi_2(\varphi_1(t),\varphi(t,s))=s$, which implies 
\[
\left[s\cdot k(\gamma t+\delta)^n+b(t)\right]\cdot k\cdot \left[\gamma\cdot \frac{\alpha t +\beta}{\gamma t+\delta}+\delta\right]^n+c(\psi(t))=s.
\]
From here, we deduce that
\[
k^2\cdot \left[\gamma(\alpha +\delta)t+(\gamma \beta+\delta^2)\right]^n=1,
\]
which in turn yields 
\[
\gamma(\alpha+\delta)=0,\mbox{ }k^2(\gamma \beta+\delta^2)^n=1.
\]
Thus, either $\alpha=-\delta$ and $k^2(\gamma \beta+\delta^2)^n=1$, or $\alpha=\delta$, $\gamma=0$ and $k^2\delta^{2n}=1$.
\end{itemize}

Putting (i) and (ii) together, we get the following result.

\begin{theorem}\label{fund3}
Let $S$ be a rational ruled surface properly parametrized as in Eq. \eqref{surf}, which is not doubly ruled, let $\bfq(t)=(q_1(t),q_2(t),q_3(t))$, and let also $n=\mbox{max}\{\deg(q_1(t)),\deg(q_2(t)),\deg(q_3(t))\}$. Let $f({\bf x})=\bfQ {\bf x}+\bfb$ be an involution of $S$, and let $\varphi:\RR^2\to \RR^2$ be the birational transformation making the diagram in Eq. \eqref{eq:fundamentaldiagram} commutative. Then $\varphi(t,s)$ is as in Eq. \eqref{phi-funct-2}, with $\psi(t)$ as in Eq. \eqref{eq:Moebius}, and: 
\begin{itemize}
\item [(I)] $\alpha=-\delta$, $k^2(\gamma \beta+\delta^2)^n=1$, or
\item [(II)] $\beta=\gamma=0$, $k^2\delta^{2n}=1$, $\alpha=\delta\neq 0$. 
\end{itemize}
\end{theorem}

\noindent In case (I) of Theorem \ref{fund3} we have dropped the number of parameters by 2, so we end up with 2 parameters at most. In case (II) of Theorem \ref{fund3} we observe that the only possibility for the function $\varphi$, besides the identity, is $\varphi(t,s)=(t,\pm s+c(t))$. 

\subsection{Computing the symmetries.} \label{computing}

In order to find the symmetries $f({\bf x})=\bfQ {\bf x}+\bfb$ of $S$ we must determine first the functions $a(t),\psi(t)$ in Eq. \eqref{phi-funct}. For this purpose, taking norms in Eq. \eqref{at} we reach the condition

\begin{equation}\label{eq-compu}
\Vert \bfq(t)\Vert^2-k^2\cdot (\gamma t+\delta)^{2n}\cdot \Vert \bfq(\psi(t))\Vert^2=0.
\end{equation}

\noindent Setting all the coefficients in $t$ of the left hand-side of Eq. \eqref{eq-compu} equal to zero, we get a polynomial system ${\mathcal P}$ of degree $2n$. This polynomial system obviously contains the variables $\gamma,\delta,k$. However, one can see that ${\mathcal P}$ contains the variables $\alpha,\beta$ as well. 

\begin{lemma} \label{noconstant}
The polynomial system ${\mathcal P}$ has degree $2n$ and contains explicitly the variables $k,\gamma,\delta,\alpha,\beta$. 
\end{lemma}

\begin{proof} We just need to show that $\alpha,\beta$ also appear in ${\mathcal P}$. The only possibility for $\alpha,\beta$ not appearing in ${\mathcal P}$ is that $\Vert \bfq(t)\Vert$ is a constant ${\bf c}$. But let us see that this cannot happen. Indeed, suppose that $\Vert \bfq(t)\Vert={\bf c}$. Then $\bfq(t)$ parametrizes a curve contained in a sphere centered at the origin, of radius ${\bf c}$. However, the curve parametrized by $\bfq(t)$ is polynomial, therefore non-bounded, so the curve parametrized by $\bfq(t)$ cannot be contained in a sphere.
\end{proof}

\begin{remark} \label{degree}
The degree of the polynomial system corresponding to the naive method described in the Introduction to the paper is $N$, the degree of the implicit equation of $S$. From Proposition 1 in \cite{Chen2}, $N\leq \mbox{deg}(\bfp(t))+n-\lambda$, where $\mbox{deg}(\bfp(t))$ is the maximum degree in $t$ of the components of $\bfp(t)$ previously reduced to the common denominator, and $\lambda$ is the degree of a certain gcd. Additionally, $\lambda=0$ unless $\bfx(t,s)$ has certain type of base points. Generically, and in particular in the absence of this type of base points, we have $N=\mbox{deg}(\bfp(t))+n$. Therefore, $\deg({\mathcal P})\leq N$ iff $\mbox{deg}(\bfp(t))\geq n$, where $n=\mbox{deg}(\bfq(t))$. In special situations, and in particular in the presence of certain base points (see Proposition 1 in \cite{Chen2}), the analysis is less clear. 
\end{remark}

Since we can always drop at least one parameter in the polynomial system ${\mathcal P}$ (for instance, assuming that one of the variables $\alpha,\beta,\gamma,\delta$ is 1), we need to deal with polynomial systems of at most 4 unknowns. In the case of involutions we can reduce the number of unknowns to 2, so we end up with a bivariate system. In our experimentation we have observed that the system ${\mathcal P}$ is usually zero-dimensional. The interested reader can check the recent paper \cite{BS16} on solving zero-dimensional polynomial systems in any number of variables; for bivariate systems, see \cite{BLMPRS16, DET09}. Although it is less frequent, it can happen that ${\mathcal P}$ has infinitely many solutions (for instance, when all the components of $\bfq(t)$ are linear), as well. For now, we will suppose that ${\mathcal P}$ has finitely many solutions; we will address the case when ${\mathcal P}$ has infinitely many solutions later. 

Now if ${\mathcal P}$ has no real solutions, then $S$ has no real symmetries. If ${\mathcal P}$ is consistent, each real solution $\alpha,\beta,\gamma,\delta,k$ of ${\mathcal P}$ corresponds to a tentative symmetry $f({\bf x})=\bfQ {\bf x}+\bfb$, where $\bfQ$ can be found from Eq. \eqref{at} by solving a linear system. Whenever $\bfQ$ is orthogonal, in order to find $\bfb$ we go back to Eq. \eqref{fundam-eq}. Taking into account that $f({\bf x})=\bfQ {\bf x}+\bfb$ and Theorem \ref{fund2}, we get

\begin{equation}\label{system}
\left\{\begin{array}{ccc}
\bfQ_{11}\cdotp p_1(t)+\bfQ_{12}\cdotp p_2(t)+\bfQ_{13}\cdotp p_3(t)+b_1 & = & p_1(\psi(t))+c(t)q_1(\psi(t)),\\
\bfQ_{21}\cdotp p_1(t)+\bfQ_{22}\cdotp p_2(t)+\bfQ_{23}\cdotp p_3(t)+b_2 & = & p_2(\psi(t))+c(t)q_2(\psi(t)),\\
\bfQ_{31}\cdotp p_1(t)+\bfQ_{32}\cdotp p_2(t)+\bfQ_{33}\cdotp p_3(t)+b_3 & = & p_3(\psi(t))+c(t)q_3(\psi(t)).
\end{array}\right.
\end{equation}

The $\bfQ_{ij}$ are the entries of the matrix $\bfQ$, and $\bfb=(b_1,b_2,b_3)$. Since at this point we know $\alpha,\beta,\gamma,\delta,k$, $a(t)$ and $\psi(t)$ are known too. However, $c(t),\bfb$ are not known yet. Now we proceed as follows: 

\begin{itemize}
\item [(i)] Multiply the first equation of Eq. \eqref{system} by $q_2(\psi(t))$. Then multiply the second equation of Eq. \eqref{system} by $q_1(\psi(t))$, and subtract both equations. This way, we reach an equation $E_1(t,b_1,b_2)=0$, linear in $b_1,b_2$. 
\item [(ii)] Proceeding in the same way with the second and third equations, we get $E_2(t,b_2,b_3)=0$, linear in $b_2,b_3$. 
\item [(iii)] Evaluate $E_1(t,b_1,b_2)=0$ and $E_2(t,b_2,b_3)=0$ at several random $t$-values. This way we get a linear system in $b_1,b_2,b_3$, whose solution provides $\bfb$.
\item [(iv)] Finally, we check that there exists $c(t)$, rational, satisfying Eq. \eqref{system}. In the affirmative case, by Theorem \ref{th-fundam}, $f({\bf x})=\bfQ {\bf x}+\bfb$ is a symmetry of $S$.
\end{itemize}

It can happen that the linear system derived in (ii) has infinitely many solutions, depending on (at most two) parameters. In that case, we can still write $c(t)$ in terms of these parameters, and go back to Eq. \eqref{fundam-eq}. Since the number of symmetries of $S$ is finite, this system is necessarily zero-dimensional, and the solutions provide the 
symmetries of $S$.

A similar situation appears in the case when the polynomial system ${\mathcal P}$ has infinitely many solutions. In that case, typically some of the parameters $\alpha,\beta,\gamma,\delta,k$ can be written in terms of the others, that we represent by $\xi_1,\ldots,\xi_p$. Even in this case, we can still compute the matrix $\bfQ$ in terms of $\xi_1,\ldots,\xi_p$ (by solving the linear system which stems from Eq. \eqref{at}, whose unknowns are the $\bfQ_{ij}$), and perform the process (i-iv) to find $\bfb$ and $c(t)$, which again are written in terms of $\xi_1,\ldots,\xi_p$. In order to find the values of $\xi_1,\ldots,\xi_p$ corresponding to symmetries of $S$, we impose that Eq. \eqref{fundam-eq} holds. This yields a zero-dimensional system whose solutions give rise to the symmetries of $S$.

\subsection{The special case of conical surfaces.} \label{subsec-conical}

We say that $S$ is a \emph{conical} surface if all the rulings of $S$ intersect at one point ${\bf p}_0\in S$, called the \emph{vertex}. The vertex can be computed by using the results in \cite{AG17}, and by applying a translation if necessary, we can always assume that ${\bf p}_0$ is the origin. Therefore, if $S$ is rational and properly parametrized we can assume that $S$ is given by means of a parametrization $\bfx(t,s)=s\bfq(t)$, where $\bfq(t)$ is polynomial. If the vertex coincides with the origin, the symmetries of $S$ have the form $f({\bf x})=\bfQ {\bf x}$, with $\bfQ$ orthogonal. Therefore, we have the following result. 

\begin{theorem}\label{th-conic} Let $S$ be a rational conical surface, properly parametrized by $\bfx(t,s)=s\bfq(t)$. Then the function $\varphi(t,s)$ in Proposition \ref{fund} has the form $\varphi(t,s)=(\psi(t),sk(\gamma t+\delta)^n)$, where $\psi(t)=\frac{\alpha t+\beta}{\gamma t+\delta}$, with $\alpha \delta-\beta\gamma\neq 0$ and $n$ equal to the maximum degree of the components of $\bfq(t)$. 
\end{theorem}

In particular, computing the symmetries of conical surfaces $\bfx(t,s)=s\bfq(t)$ is easier, since in this case the function $c(t)$ in Eq. \eqref{phi-funct} and $\bfb$ are identically zero.  

\subsection{A detailed example} \label{detailed}

We illustrate the previous ideas in the following example.
\begin{example}\label{example2}
Let $S$ be the rational ruled surface parametrized by $\pmb{x}(t,s)=\pmb{p}(t)+s\cdot \pmb{q}(t)$, with
$\bfp(t)=(p_1(t),p_2(t),p_3(t))$ and
\begin{equation}\label{p-example}
\begin{array}{rcl}
p_1(t) & = & \displaystyle{\frac{2t^8-10t^6-10t^4+5t^2+1}{t^2+1}},\\[0.4cm]
p_2(t) & = &-\displaystyle{\frac{t^9-6t^7+6t^3+t^2-3t+1)}{t^2+1}},\\[0.4cm]
p_3(t) & = & t^7+3t^5+3t^3+t+5,
\end{array}
\end{equation}
and
\begin{equation}\label{q-example}
\bfq(t)=(-t^6+7t^4-7t^2+1,2t(t^4-6t^2+1),(t^2+1)^3).
\end{equation}

\noindent In this case, $n=6$ (see Eq. \eqref{eq-n}).

The system ${\mathcal P}$, consisting of all the coefficients in $t$ of Eq. \eqref{eq-compu}, is a polynomial system of degree 12. In order to compute the solutions of ${\mathcal P}$, we consider separately the solutions with $\gamma=0$ and $\gamma\neq 0$. In the first case, since $\gamma=0$ the variable $\delta$ must be nonzero, and we can assume $\delta=1$. Here, besides the identity, we get the solutions
	$$\{\alpha=1,\beta=0, k=-1\},\; \{\alpha=-1,\beta=0, k=1\}\; \text{and} \; \{\alpha=-1,\beta=0, k=-1\},$$ 
corresponding to 
\[\varphi_1(t,s)=(t,-s+c(t)),\mbox{ }\varphi_2(t,s)=(-t,s+c(t)),\mbox{ }\varphi_3(t,s)=(-t,-s+c(t)).\]
	
If $\gamma\neq0$, we can assume $\gamma=1$. The solutions of ${\mathcal P}$ in this case are

	$$\begin{array}{rl}
	\{\alpha=0, \beta=1, \delta=0, k=1\},&\{\alpha=0, \beta=1, \delta=0, k=-1\} \\[0.2cm]
	
	\{\alpha=0, \beta=-1, \delta=0, k=1\} & \{\alpha=0, \beta=-1, \delta=0, k=-1\} \\[0.2cm]
	
	\{\alpha=1, \beta=1, \delta=-1, k=\frac{1}{8}\} &\{\alpha=1, \beta=1, \delta=-1, k=-\frac{1}{8}\}\\[0.2cm]
	
	\{\alpha=1, \beta=-1, \delta=1, k=\frac{1}{8}\}& \{\alpha=1, \beta=-1, \delta=1, k=-\frac{1}{8}\}\\[0.2cm]
	
	\{\alpha=-1, \beta=1, \delta=1, k=\frac{1}{8}\}&
	\{\alpha=-1, \beta=1, \delta=1, k=-\frac{1}{8}\}\\[0.2cm]
	
	\{\alpha=-1, \beta=-1, \delta=-1, k=\frac{1}{8}\}& \{\alpha=-1, \beta=-1, \delta=-1, k=-\frac{1}{8}\},
	\end{array}$$
	which give 12 more possibilities for $\varphi$.
	
	Now for each $\varphi(t,s)$, we need to find $f({\bf x})=\bfQ {\bf x}+\bfb$. Take for instance the solution 
	\[
	\{\alpha=\delta=0,\mbox{ }\beta=\gamma=1,\mbox{ }k=-1\},
	\]
	corresponding to $\varphi(t,s)=\left(\frac{1}{t},-st^6+c(t)\right)$. In this case, Eq. \eqref{at} corresponds to 
	\[
	\bfQ \cdot\bfq(t)=-t^6\cdot \bfq\left(\frac{1}{t}\right),
	\]
	with $\bfq(t)$ as in Eq. \eqref{q-example}, which yields a linear system for the entries of $\bfQ$. After solving this linear system, we get 
	
	\begin{equation}\label{Q-example}
	\bfQ=\left( \begin{array}{ccr}
		-1 & 0 & 0\\		
		0 & 1 & 0 \\
		0 & 0 & -1\\
		\end{array}\right), 
		\end{equation}
which corresponds to an orthogonal matrix, and in fact to an axial symmetry. In order to compute $\bfb$, we use Eq. \eqref{system} with the $p_i(t)$ in Eq. \eqref{p-example}, the $q_i(t)$ in Eq. \eqref{q-example}, the $\bfQ$ in Eq. \eqref{Q-example}, and $\psi(t)=\frac{1}{t}$. We eliminate $c(t)$ by multiplying the equations in Eq. \eqref{system} by appropriate factors and subtracting, and then we get two linear equations in $b_1,b_2,b_3$, with coefficients depending on $t$. By evaluating these equations at random values of $t$, we get a linear system whose solution corresponds to 

\begin{equation} \label{b-example}
\bfb=\left( \begin{array}{ccc}
		4 &		
	    0 &
		10
		\end{array}\right)^T.
		\end{equation}
	
Going back to Eq. \eqref{system}, we observe that $c(t)=-\frac{t^8+1}{t}$, so $f({\bf x})=\bfQ {\bf x}+\bfb$ with the $\bfQ$ in Eq. \eqref{Q-example} and the $\bfb$ in Eq. \eqref{b-example} is a symmetry of $S$. In fact, $f({\bf x})$ is a an axial symmetry with respect to the line $\{x=2,z=5\}$. 

Proceeding in the same way, we obtain two reflections, three axial symmetries and two rotations composed with reflections. The symmetry planes and symmetry axes of the surface are shown in Fig. \ref{ej}. 
	
\begin{figure}
\begin{center}
$$\begin{array}{cc}
\includegraphics[scale=0.23]{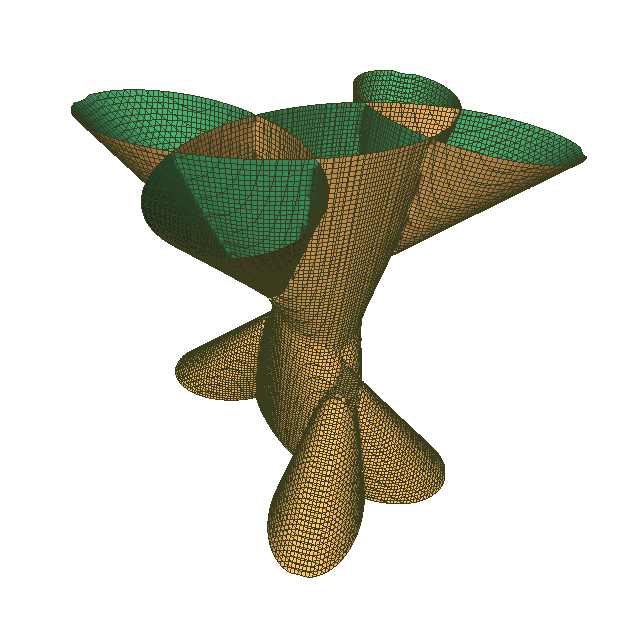} & \includegraphics[scale=0.45]{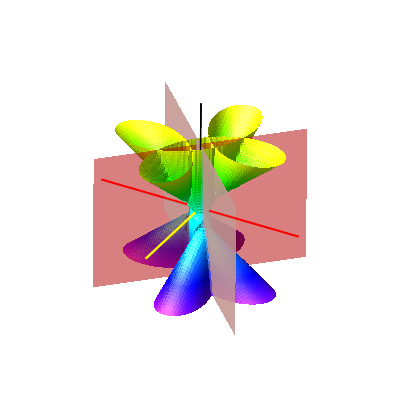}
\end{array}$$
\end{center}
\caption{The surface in Example \eqref{example2}, and its planes and axes of symmetry (right).}\label{ej}
\end{figure}
	
\end{example}

\section{Experimentation and performance of the method.}\label{sec-exp}

We have implemented the method described in Section \ref{sec-symmetries} in the computer algebra system Maple 18, and we have tried several examples in an Intel(R) Core(TM) 2, Quad CPU Q6600, with 2.40 GHz and 4 Gb RAM. We summarize the results obtained for 10 representative examples in Tables 1 and 2: for each example, we have included: (1) a picture of the surface; (2) the degree (deg) of the parametrization, i.e. the maximum power of $t$ appearing in the numerators and denominators of $\bfp(t),\bfq(t)$; (3) the computation time (in seconds) of our method for all the symmetries (``all") and only for involutions (``involutions"), and also the computation time of the naive method described in the Introduction to the paper, which uses the implicit equation (``implicit"); (4) the symmetries found. In some cases, the symmetries detected are compositions of rotations and reflections, denoted as ``rotation+reflection". Additionally, in each case we include the parametrization of the surface. The computation time of the method using the implicit equation of the surface assumes that the implicit equation is already available, i.e. it does not include the time required to compute the implicit equation itself. 

Compared to our approach, in general computing the implicit equation of the surface and applying the naive method mentioned in the Introduction to the paper provides worse timings, even if the time to compute the implicit equation is not taken into account. In fact, Maple was able of providing an answer in less than 90 seconds only in two of the examples. In the case of the surface $\bfx_7(t,s)$, the implicit equation is very simple ($F(x,y,z)=x^3-27yz^2$), so the method using the implicit equation is faster than our method. In each example, we write in red the worst timing between our algorithm for computing all the symmetries (``all") and the method using the implicit equation (``implicit"). 

We have observed that almost all the time is spent solving the polynomial system ${\mathcal P}$, arising from Eq. \eqref{eq-compu}. We used the Maple instruction {\tt solve} to find the solutions of this system. In fact, the complexity of the method is dominated by the solution of the polynomial system ${\mathcal P}$. We have also observed that ${\mathcal P}$ is usually zero-dimensional. A recent polynomial bound for solving a zero-dimensional polynomial system is given in \cite{BS16}. Although the case when ${\mathcal P}$ is not zero-dimensional is much less frequent, it can happen as well, for instance when the components of $\bfq(t)$ are linear. In this case, up to our knowledge there is no algorithm other than Gr\"obner bases to solve the problem; the best known complexity, in this case, is simply exponential \cite{F5}.

\newgeometry{left=3cm, right=3cm}

\begin{table}
	
	\begin{tabular*}{\columnwidth}{l@{\extracolsep{\stretch{1}}}*{5}{c}}
		\toprule
		parametrization & picture & deg. & \hspace{0.5cm} computation time & symmetries\\
	    &  &  & \hspace{0.65cm}all\,/\,involutions\,/ & \\
	    
	    &  &  & \hspace{0.65cm}implicit & \\
		\midrule
		&\multirow{3}{*}{\includegraphics[scale=0.09]{surf9}} & & &\\
		& & & & 3 axial \\
		$\bfx_1(t,s)$ & & 9 &\hspace{0.5cm} 9.640\,/\,6.599\,/ &  2 mirror\\ & & &\textcolor{red}{$>90$} & 2 rotational + reflect. \\& & & & \\[0.2cm]
		\multicolumn{5}{c}{\footnotesize{$\bfx_1(t,s)=\left( \frac{2t^8-10t^6-10t^4+5t^2+1}{t^2+1},-\frac{t^9-6t^7+6t^3+t^2-3t+1)}{t^2+1},t^7+3t^5+3t^3+t+5\right)$}}\\
		\multicolumn{5}{c}{\footnotesize{$+s\cdot(2t(t^4-6t^2+1),-t^6+7t^4-7t^2+1,(t^2+1)^3)$}}\\
		\midrule
		$\bfx_2(t,s)$
		& \raisebox{-3em}{\includegraphics[scale=0.18]{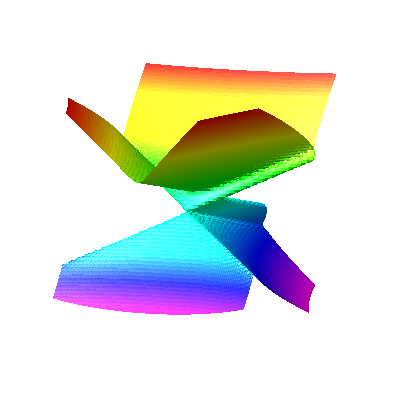}}
		& 7 & \hspace{0.5cm}\vspace{-0.9cm}10.140\,/\,4.899\,/ & 1 reflect.\\
		& & &\vspace{0.5cm} \textcolor{red}{$>90$}& \\
		\multicolumn{5}{c}{\small{$\bfx_2(t,s)=\left( \dfrac{t^7+7t^5+3t^3-t^2-3t+1}{t^2+1},\dfrac{2t(4t^5+4t^3+1)}{t^2+1},t(t^2+1)^2\right)+s\cdot(-t^4-6t^2+3,8t^3,(t^2+1)^2)$}}\\
		\midrule
		$\bfx_3(t,s)$
		& \raisebox{-3em}{\includegraphics[scale=0.17]{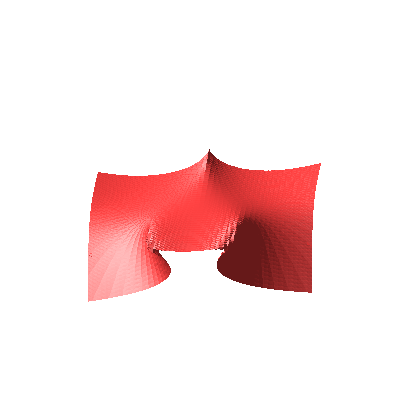}}
		& 7 & \hspace{0.5cm}\vspace{-0.9cm}1.888\,/\,1.778\,/ & 1 reflect. \\
		& & &\vspace{0.5cm} \textcolor{red}{$>90$}& \\
		\multicolumn{5}{c}{\small{$\bfx_3(t,s)=(t^6-6t^4+t^2+2t,-t^7+6t^5-t^3+t^2+t,t^3+t)+s\cdot(t^5-6t^3+t,-t^6+6t^4-t^2+1,t^2+1)$}}\\
		\midrule
		$\bfx_4(t,s)$
		& \raisebox{-3em}{\includegraphics[scale=0.21]{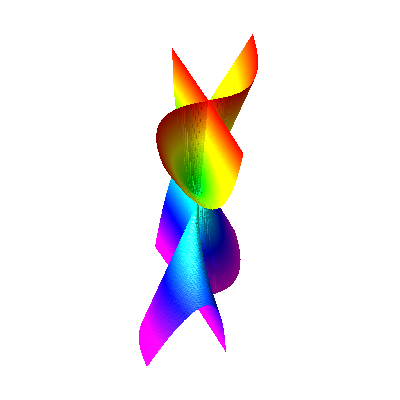}}			& 5 & \hspace{0.5cm}\vspace{-0.9cm}1.684\,/\,1.607\,/ & 1 reflect. \\
		& & &\vspace{0.5cm} \textcolor{red}{$>90$}& \\
		\multicolumn{5}{c}{$\bfx_4(t,s)=\left( \dfrac{t^2}{t^2+1},\dfrac{t^4}{t^2+1},\dfrac{t^5}{t^2+1}\right)+s\cdot(t,t^3,1)$}\\
		\midrule
		&\multirow{6}{*}{\includegraphics[scale=0.18]{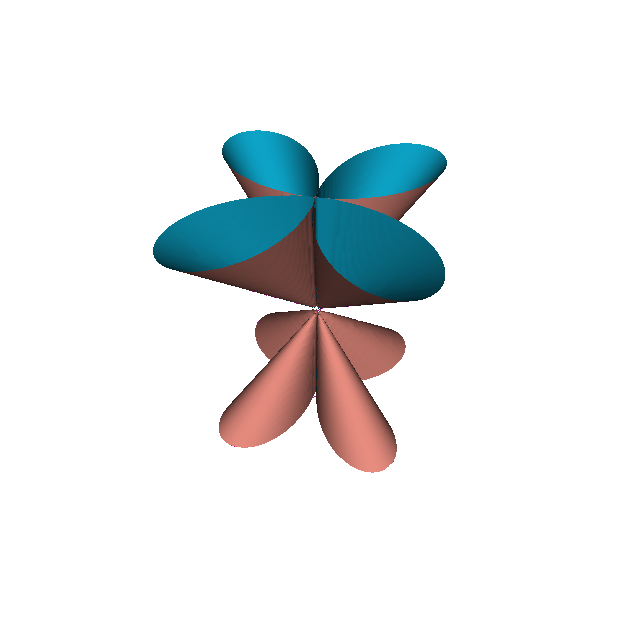}} & & & & \\
		& & & & 5 reflect. \\ & & &\hspace{0.5cm} 4.587\,/\,3.291\,/ & 5 axial sym. \\
		$\bfx_5(t,s)$ & & 6 &\textcolor{red}{$>90$} &  central\\
		& & &  & 2 rotational sym. \\  & & & & 2 rotational +  reflect. \\[0.2cm]
		\multicolumn{5}{c}{$\bfx_5(t,s)=s\cdot( 2t(t^4-6t^2+1),(-t^2+1)(t^4-6t^2+1),(t^2+1)^3)$}\\
		\bottomrule
		
	\end{tabular*}
	\caption{}
	\end{table}

\begin{table}
	
	\begin{tabular*}{\columnwidth}{cc@{\extracolsep{\stretch{1}}}cccc}
		\toprule
		parametrization & picture & deg. &\hspace{0.3cm} computation time & symmetries\\
		&  &  & \hspace{0.5cm}all\,/\,involutions & \\
		 &  &  & \hspace{0.65cm}implicit & \\
		\midrule
		$\bfx_6(t,s)$
		& \raisebox{-3em}{\includegraphics[scale=0.17]{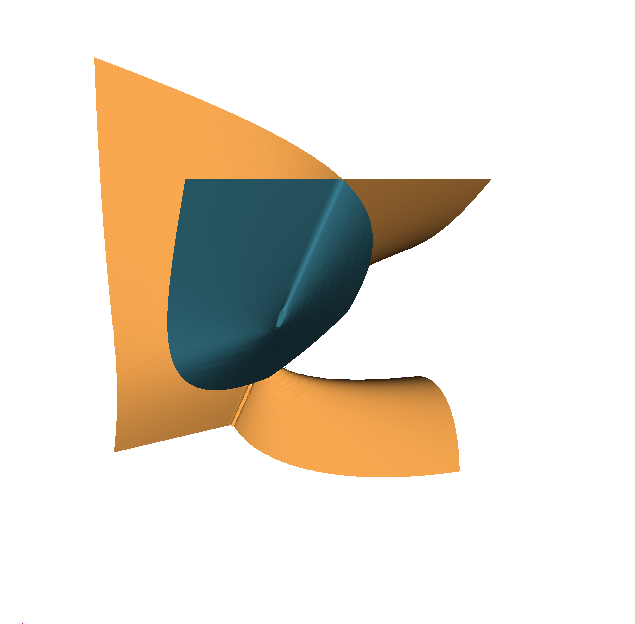}}			& 2 &\hspace{0.3cm} \vspace{-0.9cm}3.448\,/\,3.166\,/ & 1 axial sym. \\
		& & &\vspace{0.5cm} \textcolor{red}{63.648}& \\
		\multicolumn{5}{c}{$\bfx_6(t,s)=(4,1,t)+s\cdot((t+1)^2,t+1,1)$}\\
		\midrule
		& \multirow{6}{*}{\includegraphics[scale=0.18]{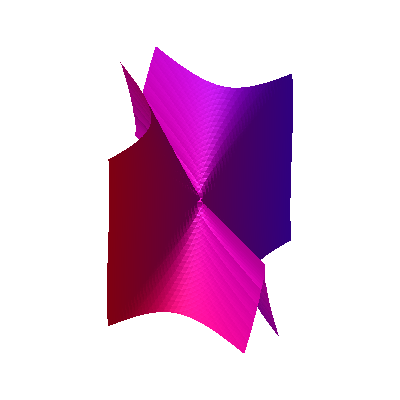}} & & &\\  & & & & central\\
		$\bfx_7(t,s)$ &	& 3 & \hspace{0.3cm}\textcolor{red}{1.451}\,/\,1.185\,/ &  1 reflection \\ & & & 0.296& 1 axial sym. \\  & & & &\\ & & & &\\
		\multicolumn{5}{c}{$\bfx_7(t,s)=s\cdot(3(t+1)^2(t-1),(t-1)^3,(t+1)^3)$}\\
		\midrule
		$\bfx_8(t,s)$
		& \raisebox{-3em}{\includegraphics[scale=0.19]{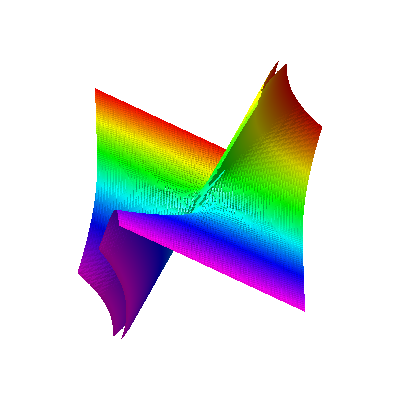}}
		& 7 &\hspace{0.3cm}\vspace{-0.9cm} 1.935\,/\,1.809\,/ & central \\
		& & &\vspace{0.5cm} \textcolor{red}{$>90$}& \\
		\multicolumn{5}{c}{$\bfx_8(t,s)=\left( \dfrac{t^3}{t^2+1},\dfrac{t^5}{t^2+1},\dfrac{t^7}{t^2+1}\right)+s\cdot(-t^5+t,3t^7,-2t^3)$}\\	
		\midrule
		$\bfx_9(t,s)$
		& \raisebox{-3em}{\includegraphics[scale=0.18]{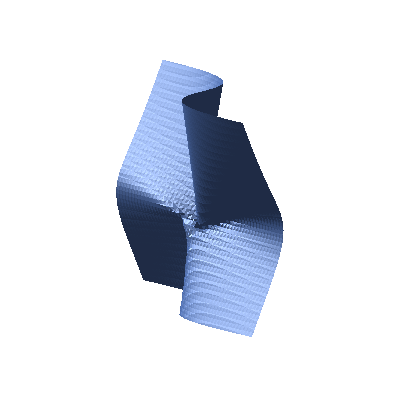}}
		& 6 &\hspace{0.3cm}\vspace{-0.9cm} 1.716\,/\,1.653\,/ & 1 reflection \\
		& & &\vspace{0.5cm} \textcolor{red}{$>90$}& \\
		\multicolumn{5}{c}{$\bfx_9(t,s)=(t^4+t^2+t,t^6+t^3,t^5+t^3+t^2+3t)+s\cdot(t^3+t,t^5,t^4+t^2+3)$}\\
		\midrule
		&\multirow{3}{*}{\includegraphics[scale=0.15]{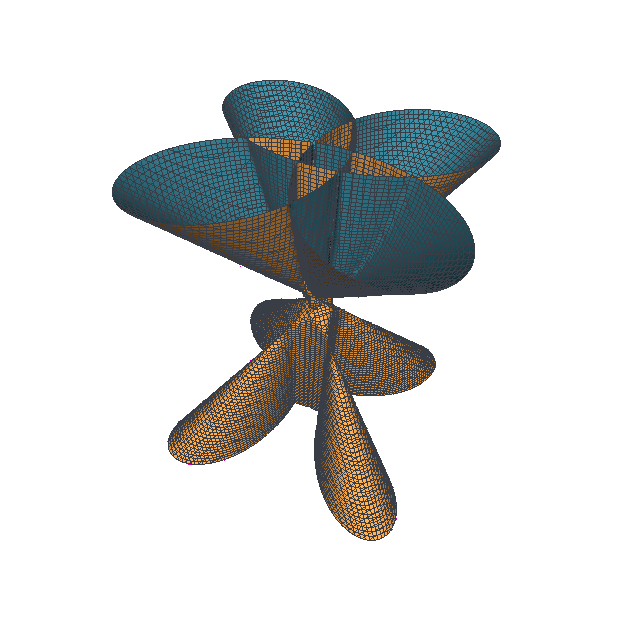}} & & &\\
		& & & & 4 reflect. \\
		$\bfx_{10}(t,s)$ & & 17 &\hspace{0.3cm} 9.828\,/\,6.973\,/ &  1 axial sym.\\ & & &\textcolor{red}{$>90$} & 2 rotational sym. \\& & & & \\[0.2cm]
 	\multicolumn{5}{c}{\footnotesize{$\bfx_{10}(t,s)=\left( -\frac{t^{17}-6t^{15}+6t^{11}-6t^7+6t^3-t^2-t+1}{t^2+1}, \frac{2t(t^{15}-5t^{13}-5t^{11}+t^9+t^7-5t^5-5t^3+t+1)}{t^2+1}, t(t^2+1)^3(t^8+1)\right) $}}\\
	\multicolumn{5}{c}{\footnotesize{$+s\cdot(-t^6+7t^4-7t^2+1,2t(t^4-6t^2+1),(t^2+1)^3)$}}\\	
		\bottomrule
		
	\end{tabular*}
	\caption{}
\end{table}
%\end{changemargin} 

\restoregeometry

\section{Application to implicit algebraic surfaces under certain conditions.} \label{app}

In this section we will see how to apply the method developed in the previous sections to find the reflections and rotational symmetries (in particular, axial symmetries) of an implicit algebraic surface, under certain conditions. Let $F(x,y,z)$ define an irreducible, implicit algebraic surface $S$ of total degree $N$, and let 
\[
F(x,y,z)=F_N(x,y,z)+F_{N-1}(x,y,z)+\cdots +F_0(x,y,z),
\]
where $F_i(x,y,z)$ denotes the homogeneous form of $F(x,y,z)$ of degree $i=0,1,\ldots,N$. Thus, $F_i(x,y,z)$ is a homogeneous polynomial of degree $i$. In particular, we refer to $F_N(x,y,z)$ as the {\it highest order form} of $F(x,y,z)$. Let ${\bf x}=(x,y,z)$, and let $f({\bf x})=\bfQ {\bf x}+\bfb$ be a symmetry of $S$. The following two lemmas show the connections of the problem treated in this section with the ideas developed in previous sections.

\begin{lemma} \label{FN}
Let $f({\bf x})=\bfQ {\bf x}+\bfb$ be a symmetry of the surface $S$ defined by $F(x,y,z)=0$, where $F$ is irreducible. Then $\widehat{f}({\bf x})=\bfQ {\bf x}$ is a symmetry of the surface defined by $F_N(x,y,z)=0$.
\end{lemma}

\begin{proof} Since $F(x,y,z)$ is irreducible by hypothesis, if $f({\bf x})=\bfQ {\bf x}+\bfb$ is a symmetry of $S$ then $F(\bfQ {\bf x}+\bfb)=\lambda F(x,y,z)$, with $\lambda$ a constant. Since 
\[
F(\bfQ {\bf x}+\bfb)=F_N(\bfQ {\bf x})+F_{N-1}(\bfQ {\bf x}+\bfb)+\cdots,
\]
we conclude that $F_N(\bfQ {\bf x})=\lambda F_N(\bf x)$, and the result follows. 
\end{proof}

\begin{lemma} \label{FN-conic}
The surface $F_N(x,y,z)=0$ is a conic surface, with vertex at the origin. 
\end{lemma}

\begin{proof} Since $F_N(x,y,z)$ is a homogeneous polynomial, for any constant $\beta$ we have $F_N(\beta x,\beta y,\beta z)=F_N(\beta {\bf x})=\beta^N F_N(x,y,z)$. Thus, for any point $(x,y,z)$ of the surface $F_N(x,y,z)$, the line connecting $(x,y,z)$ with the origin is included in the surface. 
\end{proof}

Now we can make precise the conditions under which the ideas in previous sections can be applied to the problem in this section. Whenever $F_N(x,y,z)$ (1) is irreducible and (2) defines a real, rational surface, from Lemma \ref{FN-conic} we know that $F_N(x,y,z)=0$ is a conic surface $S_N$, with vertex at the origin, whose symmetries can, therefore, be found by applying the method in this paper (in particular, using Theorem \ref{th-conic}). Additionally, if $F_N(x,y,z)=0$ defines a rational surface, since $S_N$ is a conic surface the intersection of the surface with a generic plane is a rational curve that we can parametrize rationally with well-known methods. As a consequence, a parametrization of $S_N$ of the type $\bfx(t,s)=s\bfq(t)$ can be computed. In turn, the symmetries of $S_N$ can be obtained by applying our methods, and taking Lemma \ref{FN} into account, the rotational symmetries and reflections in planes of $S$ can be found from here. The next lemma sheds some light on this last step. 

\begin{lemma} \label{axes}
\begin{itemize}
\item [(1)] If $f({\bf x})=\bfQ {\bf x}+\bfb$ represents a rotational symmetry about an axis ${\mathcal A}$, then $\widehat{f}({\bf x})=\bfQ {\bf x}$ represents a symmetry of the same kind, with axis ${\mathcal A}'$ parallel to ${\mathcal A}$ through the origin. 
\item [(2)] If $f({\bf x})=\bfQ {\bf x}+\bfb$ represents a symmetry with respect to a symmetry plane $\Pi$, then $\widehat{f}({\bf x})=\bfQ {\bf x}$ represents a symmetry with respect to a symmetry plane $\Pi'$ through the origin.
\end{itemize}
\end{lemma}

\begin{proof} We prove (1); the proof of (2) is similar, and is left to the reader. Let ${\mathcal A}$ be the rotational axis corresponding to $f$, let $P_0$ be a point on the axis ${\mathcal A}$, and let $\vec{v}\in {\Bbb R}^3$ be a vector parallel to ${\mathcal A}$. Since $f$ leaves ${\mathcal A}$ invariant, for any $\lambda \in {\Bbb R}$ we have $f(P_0+\lambda \vec{v})=P_0+\lambda \vec{v}$. Since $f({\bf x})=\bfQ {\bf x}+\bfb$, we get 
\[
f(P_0+\lambda \vec{v})=\bfQ\cdot (P_0+\lambda \vec{v})+\bfb=\bfQ\cdot P_0+\lambda \bfQ\cdot \vec{v}+\bfb=P_0+\lambda \vec{v}.
\]
Since the above equality holds for any $\lambda$, we conclude that $\bfQ\cdot \vec{v}=\vec{v}$. Thus, $\widehat{f}(\beta \vec{v})=\beta \vec{v}$ for any $\beta\in {\Bbb R}$, i.e. $\widehat{f}$ leaves the line ${\mathcal A}'$ parallel to ${\mathcal A}$ through the origin invariant. Since the nature of the symmetry depends upon the eigenvalues of $\bfQ$, and this matrix is common to $f$ and $\widehat{f}$, the result follows. 
\end{proof}

Recall that an axial symmetry is nothing else than a rotational symmetry of angle $\pi$, so Lemma \ref{axes} includes axial symmetries as well. Therefore, whenever $F_N(x,y,z)$ satisfies the hypotheses mentioned before, we can proceed as follows: 

\begin{itemize}
\item [(1)] Compute a parametrization $\bfx(t,s)=s\bfq(t)$ of the surface $S_N$ defined by $F_N(x,y,z)=0$.
\item [(2)] Compute the rotational symmetries and reflections of $\bfx(t,s)=s\bfq(t)$. 
\item [(3)] [rotational] Let ${\mathcal A}'$ be the axis of rotational symmetry of $S_N$, and let $\widehat{f}(\bf x)=\bfQ {\bf x}$ be the corresponding symmetry.
\begin{itemize}
\item [(i)] Apply a rigid motion to the surface $S$ defined by $F(x,y,z)=0$, so that ${\mathcal A}'$ is the $z$-axis.
\item [(ii)] If there exists $\bfb\in {\Bbb R}^3$ such that $f(\bf x)=\bfQ {\bf x}+\bfb$ is a rotational symmetry of $S$ about an axis ${\mathcal A}$, parallel to ${\mathcal A}'$, then for a generic plane $\Pi_{z_0}$, normal to the $z$-axis, the intersection curve $S\cap \Pi_{z_0}$ exhibits central symmetry around the point $P_0={\mathcal A}\cap \Pi_{z_0}$. Central symmetry of $S\cap \Pi_{z_0}$ can be detected with the algorithm in \cite{ALV18}.
\item [(iii)] Check whether or not $S$ has rotational symmetry with respect to the line ${\mathcal A}$ parallel to ${\mathcal A}'$ through $P_0$.
\end{itemize}
\item [(4)] [planar] Let $\Pi'$ be a symmetry plane of $S_N$, and let $\widehat{f}(\bf x)=\bfQ {\bf x}$ be the corresponding symmetry.
\begin{itemize}
\item [(i)] Apply a rigid motion to the surface $S$ defined by $F(x,y,z)=0$, so that $\Pi'$ is the $xz$-plane.
\item [(ii)] If there exists $\bfb\in {\Bbb R}^3$ such that $f(\bf x)=\bfQ {\bf x}+\bfb$ is a reflection in a plane $\Pi$, parallel to $\Pi'$, then for a generic plane $\Pi_{z_0}$, normal to the $z$-axis, the intersection curve $S\cap \Pi_{z_0}$ exhibits axial symmetry with respect to the line $\ell=\Pi\cap \Pi_{z_0}$. Axial symmetry of $S\cap \Pi_{z_0}$ can be detected with the algorithm in \cite{ALV18}.
\item [(iii)] Check whether or not $S$ is symmetric with respect to the plane parallel to $\Pi$ through $\ell$.
\end{itemize}
\end{itemize}

\begin{remark} If the surface $S_N$ defined by $F_N(x,y,z)=0$ is a surface of revolution, not a sphere\footnote{Notice that since we are assuming that $F_N(x,y,z)$ is irreducible, if $S_N$ is a sphere then $S$ is a quadric. However, the symmetries of a quadric can be found by using elementary methods in Linear Algebra}, $S_N$ has an axis ${\mathcal A}$ of revolution and has also infinitely many symmetry planes intersecting at ${\mathcal A}$. Whenever we apply an orthogonal change of coordinates mapping ${\mathcal A}$ to the $z$-axis, the proposed method is also valid in this case. 
\end{remark}

\begin{example} \label{ex-implicit} 
Let $S$ be an algebraic surface implicitly defined by $F(x,y,z)=x^6+y^5z+6x^5+14x^4+16x^3+8x^2+z^2$. The highest order form of $S$ is $F_N(x,y,z)=x^6+y^5z$. The polynomial $F_N(x,y,z)$ is irreducible and defines a rational surface (in fact a conic surface, with vertex at the origin), which can be parametrized, for instance, as 
\[
\left(x,y,\frac{-x^6}{y^5}\right).
\]
In order to compute a parametrization of the form $\bfx(t,s)=s\bfq(t)$, we intersect $S_N$ with the plane $y=2$. This yields the planar curve 
\[
\{x^6+32z=0,\mbox{ }y=2\},
\]
which can be parametrized as $\left(t,2,-\frac{1}{32}t^6\right)$. In turn, $S_N$ can be parametrized as $\bfx(t,s)=s\left(t,2,-\frac{1}{32}t^6\right)$. Using the method in Section 3, one can check that $S_N$ has symmetries with respect to the plane $x=0$, and with respect to the $x$-axis. 

In order to check whether or not $S$ is symmetric with respect to a plane parallel to $x=0$, we intersect $S$ with the plane $z=3$. The resulting curve is 
\[
\{x^6+6x^5+3y^5+14x^4+16x^3+8x^2+9=0,\mbox{ }z=3\}.
\]
In order to analyze the symmetries of this curve, we use the method in \cite{ALV18}. This method takes advantage of the fact that the laplacian operator $\Delta$ commutes with orthogonal transformations, so that the symmetries of the curve $g(x,y)=0$ are contained within the symmetries of $(\Delta g)(x,y)=0$. For $g(x,y)=x^6+6x^5+3y^5+14x^4+16x^3+8x^2+9$, taking laplacians twice we reach $360x^2+720x+360y+336=0$, which corresponds to a parabola symmetric with respect to the line $x=1$. Finally, we can easily check that $S$ is certainly symmetric with respect to the plane $x=x+1$. 

In order to check whether or not $S$ has axial symmetry with respect to a line parallel to the $x$-axis, we intersect $S$ with the plane $x=2$, to get the curve 
\[
\{y^5z+z^2+640=0,\mbox{ }x=2\}.
\]
Proceeding in a similar way, we check that this curve is symmetric with respect to the $x$-axis. Finally, we easily check that $S$ is symmetric with respect to the $x$-axis as well; in fact, one can check this in a straightforward way observing that $F(x,y,z)$ is invariant when we apply the transformation $\{x:=x,\mbox{ }y:=-y,\mbox{ }z:=-z\}$.

One can observe the symmetries of this surface in Figure \ref{fig-implicit}. The symmetry axis is shown in black; the solid sphere corresponds to the intersection point of the symmetry axis, and the symmetry plane. 

\begin{figure}
	\begin{center}
	$$\begin{array}{c}
		\includegraphics[scale=0.35]{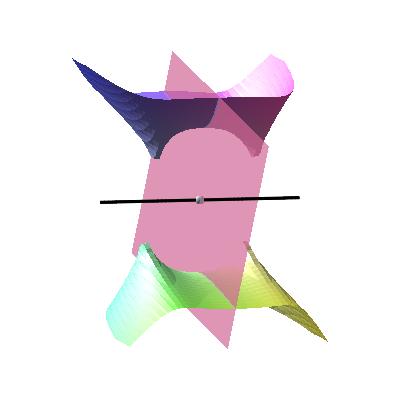} 
		\end{array}$$
	\end{center}
	\caption{Symmetry plane and symmetry axis of $F(x,y,z)=x^6+y^5z+6x^5+14x^4+16x^3+8x^2+z^2=0.$}\label{fig-implicit}
\end{figure}
\end{example}

The method proposed here, however, does not allow to find the central symmetries of the surface. For instance, one can check that the surface in Example \ref{ex-implicit} is symmetric with respect to the origin, but we cannot read this from $F_N(x,y,z)$; in fact, the form of highest degree of {\it any} polynomial in $x,y,z$ defines a surface symmetric with respect to the origin (since it is a conic surface with vertex at the origin).

\section{Conclusion.} \label{sec-conclusion}

We have presented a method to compute the symmetries of a ruled surface defined by means of a rational parametrization, working directly on rational parametric form. In order to do this, we reduce the problem to the parameter space, taking advantage of the fact that, under our hypotheses, any symmetry of the surface has an associated birational transformation of the real plane whose structure can be guessed. Additionally, the method is applicable, under certain conditions, to find some types of symmetries of implicit algebraic surfaces. The proposed method ultimately relies on polynomial system solving. However, the polynomial systems derived from our methods are much more manageable than the polynomial systems derived from naive methods to compute the symmetries. In the future, we plan to extend our ideas to other rational surfaces with structure, as well as to implicit algebraic surfaces under more general conditions.

\end{document}